\documentclass[11pt,reqno]{amsart}
\usepackage{latexsym}
\usepackage{hyperref}
\usepackage{amssymb,amsfonts,amsmath,mathrsfs}
\usepackage{geometry}
\usepackage{verbatim}
\newcommand{\eps}{\varepsilon}
\newcommand{\pr}[1]{\left( #1\right)}
\newcommand{\pg}[1]{\left\{ #1\right\}}
\newcommand{\pmd}[1]{\left| #1\right|}
\let\originalleft\left
\let\originalright\right
\renewcommand{\left}{\mathopen{}\mathclose\bgroup\originalleft}
\renewcommand{\right}{\aftergroup\egroup\originalright}
\newcommand{\es}[1]{\begin{equation}\begin{split}#1\end{split}\end{equation}}
\newcommand{\est}[1]{\begin{equation*}\begin{split}#1\end{split}\end{equation*}}
\newcommand{\R}{\mathbb{R}}

\newcommand{\C}{\mathbb{C}}
\newcommand{\N}{\mathbb{N}}
\newcommand{\Z}{\mathbb{Z}}
\renewcommand{\mod}[1]{~\pr{\textnormal{mod}~#1}}
\newcommand{\e}[1]{\operatorname{e}\pr{ #1}}

\def\sumh{\operatornamewithlimits{\sum\nolimits^h}}
\renewcommand{\d}{\textnormal{d}}
\newtheorem{theorem}{Theorem}[section]

\newtheorem{corollary}{Corollary}[section]
\newtheorem{lemma}{Lemma}[section]

\newtheorem{remark}{Remark}[section]
\newtheorem*{remark*}{Remark}

\numberwithin{equation}{section}

\begin{document}

\title{The first moment of twisted Hecke $L$-functions with unbounded shifts}
\author[S. Bettin]{Sandro Bettin}
\address{Dipartimento di Matematica, Universit\`a di Genova\\
via Dodecaneso 35\\ 16146 Genova, Italy. 
}
\email{bettin@dima.unige.it}

\begin{abstract}
We compute the first moment of twisted Hecke $L$-functions of prime power level going to infinity, uniformly in the conductor of the twist and in the vertical shift.
\end{abstract}

\maketitle

\section{Introduction}

$L$-functions associated to modular forms have been studied extensively with applications in many directions of number theory. In this paper we focus on averages of Hecke $L$-functions twisted by a primitive Dirichlet character $\chi$ of conductor coprime with the level $N$. The (twisted) $L$-functions associated to primitive forms of a given weight form an orthogonal family in the sense of Katz and Sarnak~\cite{KS}. Thus, for a primitive Dirichlet character $\chi$ with conductor $q$ coprime with $N$, one expects that
\es{\label{mccuf} 
\sumh_{f\in H_k^*(N)} L\pr{1/2,f\otimes\chi}^r = P_{r,k,\chi}(\log N)+o_{q,k,r}(1),
}
as $N\rightarrow\infty$, where $P_{r,k,\chi}$ is a polynomial of degree $\frac{r(r-1)}2$. Here, $H_k^*(N)$ denotes the subset of $H_k(N)$ consisting of primitive forms, where $H_k(N)$ is the Hecke basis for  $S_k(N)$, with $S_k(N)$ being the space of primitive cusp forms of weight $k$ and relative to the subgroup $\Gamma_0(N)$. Also, the $L$-function $L\pr{s,f\otimes\chi}$ is normalized to have central point at $s=\frac12$, that is if $f(z)$ has Fourier expansion $\sum_{n\geq1}a_n(f)n^{(k-1)/2}$ then
\est{
L\pr{s,f\otimes\chi}:=\sum_{n\geq 1}a_n(f)\chi(n)n^{-s},\qquad \Re(s)>1.
}
Finally, $\sumh$ indicates the harmonic average, that is
\est{
\sumh_{f\in H_k(N)^*}\alpha_f=\sum_{f\in H_k(N)^*}\frac{\alpha_f}{2\pi (f,f)}
}
where $(f,g)$ is Petersson's inner product.
 
Duke~\cite{Duk} computed the asymptotics~\eqref{mccuf} in the case $r=1,2$, provided that $N$ is prime and $k=2$, with an error term of size $O_\eps\pr{N^{-1/2+\varepsilon}}$. For the first moment, Ellenberg~\cite{Ell} improved the bound for the error term to $O\pr{N^{-1+\varepsilon}}$. He needed this better estimate to tackle the problem of finding all primitive solutions to the generalized Fermat equation $a^2+b^2=c^p$.

In the pioneering work~\cite{IS}, Iwaniec and Sarnak studied the first and second moment (both in the level and the weight aspects) in the case of real characters. They showed that  for $r=1,2$ the asymptotics~\eqref{mccuf} holds for all even $k\geq2$ and they relaxed also the condition on the primality of $N$, replacing it by $\frac{\varphi(N)}{N}\rightarrow 1$ with $N$ square-free,  where $\varphi(n)$ is Euler's totient function. They studied this asymptotic in an attempt to show that there are no Siegel zeros, proving that the non-existence of such exceptional zeros would follow from the non-vanishing (with some additional lower bound) of strictly more than $\frac14$ of the central values of the Hecke $L$-functions (asymptotically, when either the level or the weight goes to infinity). %They approached the problem by considering mollified first and second moment for Hecke $L$-functions twisted by real characters, but unfortunately they were able to reach but not surpass the limit of $\frac14$.

The asymptotics for the (mollified) fourth moment was proved by Kowalski, Michel and VanderKam~\cite{KMV} for prime levels. From this result they also deduced the non-vanishing of a positive proportion of the central values of $L\pr{s,f}L\pr{s,f\otimes\chi}$ for any fixed characters $\chi$. (For other applications of results on moments of Hecke $L$-functions see, among others,~\cite{DFI},~\cite{KM} and~\cite{Van}.) Their work was later extended to prime powers by Balkanova~\cite{Bal1}.
Finally, the asymptotic for the third moment was proven by Rouymi~\cite{Rou} in the case where the level is a prime power.

Rather than computing moments at the central point, it is often useful to add shifts and consider
\est{
 \sumh_{f\in\mathcal H^*_k(N)} L\pr{1/2+\alpha_1,f\otimes\chi}\cdots L\pr{1/2+\alpha_m,f\otimes\chi},
}
as these reveal more clearly the combinatorics behind the main terms. Usually the shifts are taken to be fixed (or less than $q^\varepsilon$ for some small $\varepsilon>0$), however when studying the $n$-correlation of zeros one would like to apply conjectures on moments of ratios of shifted $L$-functions and integrate over the shifts. Thus, one needs to understand for what range of shifted parameters the asymptotics for the moments still hold. 

In this paper we shall consider the shifted first moment. Kamiya addressed this problem in~\cite{Kam}, showing that if $N$ is prime and $\Re(\alpha)=0$ then
\es{\label{asymcufs}
\sumh_{f\in\mathcal H_k^*(N)} L\pr{\frac12+\alpha,f\otimes\chi}\sim 1,
}
for $k\in\pg{2,4,6,8,10,14}$ and  $q T\ll N^{\frac12-\varepsilon}$, where $T:=1+|\Im(\alpha)|$. The following theorem extends the range of validity of~\eqref{asymcufs} to $qT \ll N^{2-\varepsilon}$ with $N$ a prime power, as well as allowing for a twist of the form $a_m(f)$ as needed for non-vanishing applications~\cite{Bal2}. We take $k=2$ for simplicity, however the result is easily generalizable to all $k$. 

\begin{theorem}\label{cuorecuf}
Let $N=p^\nu$ with $p$ prime and $\nu\geq1$ and let $\chi$ be a primitive character modulo $q$ with $(q,N)=1$. Let $|\Re(\alpha)|\ll\frac1{\log N}$ and write $T=1+|\Im(\alpha)|$. Then, if $\nu\geq 2$ and $p|m$ then $\mathcal M_m(\alpha,\chi;N)=0$. In all other cases for all $\eps>0$ we have
\est{%\label{mtfcuf}
\mathcal M_m(\alpha,\chi;N)=\frac{\chi(m)}{m^{\frac12+\alpha}}(1-\delta_\nu(p))+O_\eps({(qTm)^{1/2}}N^{-1+\eps}),
}
as $N$ goes to infinity, where $\delta_\nu(p)=0$ if $\nu=1$, $\delta_\nu(p)=\frac1{p-p^{-1}}$ if $\nu=2$ and $\delta_{\nu}(p)=\frac1{p}$ otherwise.
\end{theorem}

The proof is rather simple and is based on Petersson's formula and on the functional equation for the ``twisted periodic zeta-function'' which is the meromorphic continuation to $\C$ of
\es{\label{fus}
F^*\pr{s,\chi,\frac ac}:=\sum_{n\geq1}\frac{\chi(n) \e{\frac{na}c}}{n^s}\qquad \Re(s)>1
}
with $(a,c)=1$, $c>0$, and $\chi$ a primitive character modulo $q$. Analogously to what happens in the case where $q=1$, the functional equation relates $F^*\pr{s,\chi,\frac ac}$ with $F_*\pr{1-s,\overline \chi,-aq/c}$ where $F_*\pr{s,\chi,x}$ is the ``twisted Hurwitz zeta-function''
\es{\label{fus2}
F_*\pr{s,\chi,x}:=\sum_{n+x>0}\frac{\chi(n)}{(n+x)^s},\qquad \Re(s)>1,\ x\in\R.
}
\section*{Acknowledgments}
A weaker version of Theorem~\ref{cuorecuf} was proven in the author's PhD thesis. 

The author would like to thank Olga Balkanova~for useful comments.

\section{Preliminaries and the computation of the main term}\label{prel}
\begin{remark}
Throughout  the  paper,  we  use  the  common  convention  in  analytic  number theory  that $\eps$ denotes  an  arbitrarily  small  positive  quantity  that  may  vary  from  line  to line.
\end{remark}

We define $T:=|\Im(\alpha)|+1$ and assume that $T,m,q\ll N^{100}$ (otherwise the result is trivial) and $\Re(\alpha)\ll \frac1{\log N}$.

We shall show that 
\es{\label{rmr}
M_m(\alpha,\chi; N):=\sumh_{f\in H_k(N)}{a_f(m)} L\pr{\tfrac{1}{2}+\alpha,f\otimes\chi}=\frac{\chi(m)}{m^{\frac12+\alpha}}+O_\eps\Big(\frac{{(qTm)^{1/2}}}{N^{1-\eps}}\Big),
}
where $N$ is any integer. If $N$ is prime and $k=2$, then $H_2(N)=H_2^*(N)$ and so we obtain Theorem~\ref{cuorecuf} in the case $\nu=1$.

Next, we express $L\pr{\frac{1}{2}+\alpha,f\otimes\chi}$ as a sum of length $Y\gg q^2T^2N^{1+\eps}$.

\begin{lemma}\label{fadr}
Let $f\in H_k\pr{N}$ and let $\chi$ be a primitive Dirichlet character modulo $q$ with $(q,N)=1$. Let $\eps>0$ and let $Y\gg q^2T^2N^{1+\eps}$. Then
\es{\label{apfeqcuf}
L\pr{\tfrac12+\alpha,f}&=\sum_{n\geq1}\frac{\chi(n)a_f(n)}{n^{\frac12+\alpha}}V\pr{\frac{n}{Y}}+O_{\eps,A}(N^{-A})\\
}
for any $A>0$, where
\es{\label{defvacuf}
V(x)&:=\frac1{2\pi i}\int_{(2)}e^{s^2}x^{-s}\,\frac{\d s}s.
}
\end{lemma}
\begin{proof}
Exchanging the order of summation and integration and moving the line of integration to $-M$ for some $M>0$ we see that the sum on the right hand side of~\eqref{apfeqcuf} is equal to 
\est{
L\pr{\tfrac12+\alpha,f\otimes\chi}+\frac1{2\pi i}\int_{(-M)}e^{s^2}  L\pr{\tfrac12+s+\alpha,f\otimes\chi} Y^s\,\frac{\d s}s.
}
By the functional equation
\est{%\label{feqtcf}
\Lambda(s,f\otimes\chi)&:=\pr{{\sqrt N q}/{2\pi}}^s\Gamma\pr{s+({k-1})/{2}}L(s,f\otimes\chi)\\
&\,=\omega \Lambda(1-s, f\otimes\overline\chi),
}
where $|\omega|=1$ we see that the integral is bounded by $( N q^2 T^2/Y)^{M}$,
since
\est{
\frac{\Gamma\pr{1-\alpha-s}}{\Gamma\pr{1+\alpha+ s}}\ll T^{-2\Re(s)}e^{\frac{|s|^2}{2}}
}
by Stirling's formula. The Lemma then follows by taking $M$ large enough.
\end{proof}
\begin{lemma}[Petersson's formula]
Let $\mathcal F$ be an orthonormal basis of $S_2(N)$. Then, for $m,n\geq1$ we have
\es{\label{peterf}
\sumh_{f\in\mathcal F}\overline{a_f(m)}\,a_f(n)&=\delta_{m,n}+2\pi i^{-k} \sum_{\substack{c\geq1,\\N|c}}\frac{S(m,n;c)}{c}J_{1}\pr{\frac{4\pi\sqrt{mn}}{c}},\\
}
where $\delta_{m,n}=1$ if $m=n$ and $\delta_{m,n}=0$ otherwise.
\end{lemma}

Applying Lemma~\eqref{fadr} with $Y=(mN q T)^2$ and using Petersson's formula we can write $M_m(\alpha,\chi; N)$ as
\es{\label{essrt}
&M_m(\alpha,\chi; N)=\sum_{n\geq1}\frac{\chi(n)}{n^{\frac12+\alpha}}\bigg(\sumh_{f\in F}\overline{a_f(m)}\, a_f(n)\bigg)V\pr{\frac{n}{Y}}\\
&\hspace{2.5em}=\frac{\chi(m)}{m^{\frac12+\alpha}}V\pr{\frac{m}{ Y}}-2\pi \sum_{n\geq1}\frac{\chi(n)}{n^{\frac12+\alpha}} \sum_{\substack{c\geq1,\\N|c}}\frac{S(m,n;c)}{c}J_{1}\pr{\frac{4\pi\sqrt{mn}}{c}}V\pr{\frac {n}{ Y}},
}
where $S(m,n,c)$ is the Kloosterman sum. Now,
\est{%\label{mt}
V\pr{\frac{m}{ Y}}=\frac1{2\pi i}\int_{(2)}e^{s^2}({Y/m})^{s}\,\frac{\d s}s= 1+O(N^{-A})
}
for any $A>0$, so we just need to bound the series on the last line of~\eqref{essrt}. By Weil's bound, $S(m,n,c)\ll_\eps (m,n,c)^{\frac12}c^{\frac12+\eps}$, and the bounds $J_1(x)\ll x$ and $V(x)\ll_A\min(1,x^{-A})$ for any $A>0$, the contribution to the aforementioned series coming from the $c>C$ is bounded by
\est{
\sum_{c>C}\frac{1}{c^{\frac32-\eps}}\sum_{n\geq1}\frac{m^\frac12}{n^{\Re(\alpha)}}  \pmd{V\pr{\frac {n}{Y}}}\ll \frac{m^\frac12Y}{C^{\frac12-\eps}}.
}
Taking $C=N^{D}$ with $D$ fixed but large enough, we obtain that the contribution of these terms is $O_A(N^{-2})$. Thus, opening the Kloosterman sum and exchanging the order of summation, we arrive to
\es{\label{stptcuf}
M_m(\alpha,\chi; N)&=\frac{\chi(m)}{m^{\frac12+\alpha}}-2\pi \sum_{\substack{c\leq C,\\N|c}} \sum_{\substack{a\mod {c},\\ (a,c)=1}}\frac{\e{{m\overline a/c}}}{c}T_{m}(a,c,\alpha,\chi;Y)+O(N^{-2}),
}
where 
\es{\label{refae}
T_{m}(a,c,\alpha,\chi;Y):=\sum_{n\geq1}\frac{\chi(n) \e{{na}/{c}}}{n^{\frac12+\alpha}}J_{1}\pr{\frac{4\pi\sqrt{mn}}{c}}V\pr{\frac {n}{Y}}.
}

\section{The twisted periodic zeta function}
In order to bound $T_{m}(a,c,\alpha,\chi)$ we need some properties of the twisted periodic zeta function $F^*\pr{s,\chi,\frac ac}$ defined in~\eqref{fus}.
\begin{lemma}%\label{rar}
Let $(a,c)=1$ and let $\chi$ be a primitive Dirichlet character modulo $q$. Then $F\pr{s,\chi,\frac ac}$ is an entire function of $s$ with the exception of a simple pole at $s=1$ of residue $\overline\chi(a)\frac{\tau(\chi)}q$ if $c=q$, where $\tau(\chi)$ is the Gauss sum. Moreover $F^*\pr{1-s,\chi,\frac ac}$ satisfies the functional equation
\es{\label{fee}
F^*\pr{1-s,\chi,\frac ac}&=\Gamma(s)\frac{\tau(\chi)}{q^{1-s}}\Big(e^{- \frac {\pi i s}2}F_*\pr{s,\overline \chi,-\frac{aq}{c}}+\chi(-1)e^{ \frac {\pi i s}2}F_*\pr{s,\overline \chi,\frac{aq}{c}}\Big),\\
}
where $F_*\pr{s, \chi,x}$ is as defined in~\eqref{fus2}.
\end{lemma}
\begin{proof}
We start by decomposing $F\pr{s,\chi,\frac ac}$ into a linear combination of Hurwitz's zeta functions,
\est{%\label{decfcuf}
F^*\pr{s,\chi,\frac ac}=\frac1{(cq)^s}\sum_{\ell=1}^{cq}\chi(\ell)\e{\frac{\ell a}{c}}\zeta\pr{s,\frac\ell {cq}},
}
where for $\Re(s)>1$ the Hurwitz zeta function is defined by $\zeta(s,x):=\sum_{n+x>0}(n+x)^{-s}$. The Hurwitz zeta-function is holomorphic on $\C$ with the exception of a simple pole of residue $1$ at $s=1$. Thus, $F^*\pr{s,\chi,\frac ac}$ is entire apart from (possibly) a simple pole at $s=1$. The residue is
\est{%\label{decfcuf}
\frac1{cq}\sum_{\ell=1}^{cq}\chi(\ell)\e{\frac{\ell a}{c}}=\frac1{cq}\sum_{\ell_1=1}^{q}\chi(\ell_1)\sum_{\ell_2=0}^{c-1}\e{\frac{(\ell_1+\ell_2q) a}{c}}
}
and so it is $0$ unless $c|q$ in which case it is equal to
\est{
\frac1{q}\sum_{\ell=1}^{q}\chi(\ell)\e{\frac{\ell a}{c}}=\frac1{q}\overline\chi(aq/c)\tau(\chi),
}
by~(3.12) of~\cite{IK} (and the following remark). It follows that the residue is $\overline\chi(a)\tau(\chi)/q$ if $q=c$ and otherwise $F^*\pr{s,\chi,\frac ac}$ is entire.
 
The functional equation for the Hurwitz zeta function expresses $\zeta(1-s,x)$ in terms of the periodic zeta-function $F(s,x):=\sum_{n\geq1}{\e{nx}}{n^{-s}}$:
\est{
\zeta(1-s,x)=\Gamma(s)(e^{-\frac {\pi i s}2}F(s,x)+e^{\frac {\pi i s}2}F(s,-x)).
}
Thus, for $\Re(s)<0$ we have
\est{
F\pr{1-s,\chi,\frac ac}&=\frac{\Gamma(s)}{(cq)^{1-s}}\sum_{\ell=1}^{cq}\chi(\ell)\e{\frac{\ell a}{c}}\pr{e^{-\frac {\pi i s}2}F\pr{s,\frac \ell{cq}}+e^{\frac {\pi i s}2}F\pr{s,-\frac \ell{cq}}}\\
&=\frac{\Gamma(s)}{(cq)^{1-s}}\sum_{\epsilon=\pm1}e^{\epsilon \frac {\pi i s}2}\sum_{n\geq1}\frac{1}{n^s}\sum_{\ell=1}^{cq}\chi(\ell)\e{\frac{\ell aq-\epsilon n}{cq}}.
}
The inner sum is equal to $0$ unless $\epsilon n\equiv aq\mod c$ and so
\est{
F\pr{1-s,\chi,\frac ac}&=\frac{\Gamma(s)}{c^{-s}q^{1-s}}\sum_{\epsilon=\pm1}e^{\epsilon \frac {\pi i s}2}\sum_{\substack{n= \epsilon aq +r c,\\ n>0,\ r\in\Z}}\frac{1}{n^s}\sum_{\ell^\prime=1}^{q}\chi(\ell^\prime)\e{\frac{\ell^\prime a}{c}}\e{-\epsilon\frac{n\ell^\prime}{cq}}\\
&=\frac{\Gamma(s)}{c^{-s}q^{1-s}}\sum_{\epsilon=\pm1}e^{\epsilon \frac {\pi i s}2}\sum_{\substack{n= \epsilon aq +r c,\\ n>0,\ r\in\Z}}\frac{1}{n^s}\sum_{\ell^\prime=1}^{q}\chi(\ell^\prime)\e{-\frac{\epsilon r\ell^\prime }{q}}\\
%&=\frac{\Gamma(s)}{c^{-s}q^{1-s}}\sum_{\epsilon=\pm1}\sum_{\substack{n\equiv \epsilon aq \mod c}}\frac{e^{\epsilon \frac {\pi i s}2}}{n^s}\sum_{\ell^\prime=1}^{q}\chi(-\epsilon c)\overline \chi(n)\chi(\ell^\prime)\e{\frac{ \ell^\prime }{q}}\\
&=\frac{\Gamma(s)}{q^{1-s}}\sum_{\epsilon=\pm1}e^{\epsilon \frac {\pi i s}2}\sum_{\substack{r+\epsilon aq/c >0}}\frac{1}{(r+\epsilon aq/c)^s}\overline \chi(-\epsilon r)\tau(\chi)\\
}
by~(3.12) of~\cite{IK}. Equation~\eqref{fee} then follows.
\end{proof}
From the functional equation we can obtain the following ``convexity bound'' for $F\pr{s,\chi,\frac ac}$.
\begin{corollary}
Let $\chi$ be a primitive character modulo $q$ and let $(a,c)=1$. Let $-1\leq \Re(s)\leq 1$ and let $|s-1|>\eps$ for some $\eps>0$. Then
\est{
F^*\pr{s,\chi,\frac ac}\ll_{\eps} (q+q|s|)^{\frac 12-\frac12\Re(s)+\eps}\times
\begin{cases}
1 & \text{if $c\mid q$,}\\
\pr{\pg{\frac{qa}{c}}^{-1+\Re(s)-\eps}+\pg{-\frac{qa}{c}}^{-1+\Re(s)-\eps}} & \text{if $c\nmid q$,}
\end{cases}
}
where $\{x\}$ denotes the fractional part of $x$. In particular, if $r\geq1$ and $\eta_1,\dots,\eta_{cr}\ll1$ then we have 
\es{\label{mbbe}
\sum_{\substack{a=1,\\ (a,c)=1}}^{cr} \eta_aF^*\pr{s,\chi,\frac ac} \ll_{\eps} rc^{1+\eps} (q+q|s|)^{\frac 12-\frac12\Re(s)+\eps}.
}
\end{corollary}
\begin{proof}
Let $\delta_{c,q}=1$ if $c=q$ and $\delta_{c,q}=0$ otherwise. Then for $\Re(s)=1+\eps$ with $\eps>0$ we have
\est{
F^*\pr{s,\chi,\frac ac}-\delta_{c,q}\frac{\overline\chi(a)\tau(\chi)/q}{s-1}\ll_\eps 1.
}
By the functional equation~\eqref{fee}, for $\Re(s)=-\eps$ we have
\est{
F^*\pr{s,\chi,\frac ac}-\delta_{c,q}\frac{\overline\chi(a)\tau(\chi)/q}{s-1}&\ll1+(q|s|)^{\frac12+\eps}\pr{\zeta\pr{1+\eps,-\frac{aq}{c}}+\zeta\pr{1+\eps,\frac{aq}{c}}},\\
&\ll(q|s|)^{\frac12+\eps}\pr{\pg{\frac{qa}{c}}^{-1-\eps}+\pg{-\frac{qa}{c}}^{-1-\eps}}\\
}
if $c\nmid q$ and $\ll (q|s|)^{\frac12+\eps}$ otherwise.
The Corollary then follows by the Phragm\'en-Lindel\"of theorem.
\end{proof}

\section{Bounding the error terms}
We start by recalling the Mellin transform of $J_{1}(x)$,
\est{
J_1(x)=\frac1{2\pi i }\int_{\pr{-\delta}}2^{s-1}\frac{\Gamma\pr{\frac{s+1}{2}}}{\Gamma\pr{\frac32-\frac s2}}x^{-s}\,\d s,
}
for any $-1<\delta<0$. We take $\delta=-1+\eps$ for some small $\eps>0$ and obtain
\est{
T_\epsilon(a,c,\alpha,\chi;Y)=
\frac{1}{4\pi i }\int_{\pr{-1+\eps}}\frac{\Gamma\pr{\frac{s+1}{2}}}{\Gamma\pr{\frac32-\frac s2}}\pr{\frac{2\pi}{c}}^{-s} \sum_{n\geq1}\frac{\chi(n) \e{\frac{na}{c}}}{m^\frac s2n^{\frac12+\alpha+\frac s2}}  V\pr{\frac {n}{Y}}
\d s.
}
Now, using the integral representation~\eqref{defvacuf} of $V(x)$, we have
\est{
T_m(a,c,\alpha,\chi;Y)=
&\frac{1}{2(2\pi i)^2 }\int_{\pr{-1+\eps}}\frac{\Gamma\pr{\frac{s+1}{2}}}{\Gamma\pr{\frac32-\frac s2}}\pr{\frac{2\pi}{c}}^{-s} \int_{(2)} e^{w^2}  m^{-\frac s2}\times\\
&\qquad\quad\times F\pr{\frac12+\alpha+\frac s2+w,\chi,\frac a{c}}  Y^{w}\frac{\d w}w\,\d s.
}
We move the line of integration of the $w$-integral to $\Re(w)=\eps$ without passing through any pole (notice that we can assume $c\neq q$ since we have $N|c$ and $(N,q)=1$). Thus, using~\eqref{mbbe} we obtain
\est{
&\sum_{\substack{a\mod {c},\\ (a,c)=1}} \e{\frac {m\overline a}c} T_m(a,c,\alpha,\chi;Y)\\[-1em]
&\hspace{3em}\ll (qm)^{\frac12}(cY)^\eps
\int_{\pr{\eps}}\int_{(\eps)} \frac{(|s|+|w|+|\alpha|)^\frac12 |e^{w^2} \Gamma\pr{\frac{s}{2}}|}{|\Gamma\pr{2-\frac s2}|}\frac{|\d w\,\d s|}{|w|}\ll (qmT)^\frac12 (Yc)^\eps,
}
by Stirling's formula. 
Thus, by~\eqref{stptcuf} we have
\est{%\label{stptcuf}
M_m(\alpha,\chi; N)&=\frac{\chi(m)}{m^{\frac12+\alpha}}+O_\eps\pr{{(qmT)^\frac12}/{N^{1-\eps}}},
}
as desired.% (and thus of Theorem~\ref{cuorecuf} in the prime case) is completed.

\section{Prime powers}
We now consider the case of $N=p^{\nu}$ with $\nu\geq2$.
\begin{lemma}\label{Petpp}
Let  $N=p^{\nu}$ with $p$ prime and $\nu \geq 2$. Let $\delta_{\nu}(p)=\frac1{p-p^{-1}}$ if $\nu=2$ and $\delta_{\nu}(p)=p^{-1}$ otherwise. Then, if $(p,mn)=1$ we have
\est{\sumh_{f\in H_{2k}^{*}(N)}
\lambda_f(m)\lambda_f(n)=\sumh_{f\in H_{2k}(N)}
\lambda_f(m)\lambda_f(n)-\delta_\nu(p)\sumh_{f\in H_{2k}(N/p)}
\lambda_f(m)\lambda_f(n),
}
whereas if $(p,mn)>1$ then the left hand side is equal to $0$.
\end{lemma}
\begin{proof}
This is Remark~4 of~\cite{Rou}.
\end{proof}

By Lemma~\ref{fadr} and Lemma~\ref{Petpp} we obtain that if $N=p^{\nu}$ with $\nu\geq2$ then $\mathcal M_m(\alpha,\chi;N)=0$ if $(m,p)>1$ and otherwise
\es{\label{ppf}
\mathcal M_m(\alpha,\chi;N)&=\sumh_{f\in H_{2k}(N)}a_f(m)L\pr{1/2+\alpha,f\otimes \chi}+{}\\
&\quad-\delta_\nu(p)\sumh_{f\in H_{2k}( N/p)}a_f(m)L\pr{1/2+\alpha,f\otimes \chi}+{}\\
&\quad+\frac{\chi(p)}{p^{\frac12+\alpha}}(E_{m,p}(\alpha,\chi;N)-\delta_\nu(p)E_{m,p}(\alpha,\chi;N/p))+O(N^{-2}),
}
where
\est{
E_{m,p}(\alpha,\chi;N):=-\sum_{\substack{n\geq1}}\sumh_{f\in H_{2k}(N)}\frac{\chi(n)a_f(m)a_f(pn)}{n^{\frac12+\alpha}}V\pr{\frac {pn}Y}\\
}
with $Y=(qTN)^2$. By~\eqref{rmr}, the first two terms on the right hand side of~\eqref{ppf} are equal to $\frac{\chi(m)}{m^{\frac12+\alpha}}(1-\frac1{p_{\nu}})+O({(qTm)^{1/2}}N^{-1+\eps})$, thus we just need to bound the contribution of $E_{m,p}(\alpha,\chi;N)$. Applying Petersson's formula~\eqref{peterf} and proceeding as in Section~\ref{prel} we obtain
\est{
E_{m,p}(\alpha,\chi;N):=2\pi \sum_{\substack{c\leq C,\\N|c}}\frac1c\sum_{\substack{a\mod {c},\\ (a,c)=1}}\e{\frac{m\overline a}{c}}T_{m/ p}(a,c/p,\alpha,\chi;Y/p)\\
}
with $T_{m}(a,c,\alpha,\chi;Y)$ as in~\eqref{refae} and $C=N^D$ for some large but fixed $D$. By the same arguments as in the previous section (using~\eqref{mbbe} with $r=p$) we have
\est{
\sum_{\substack{a\mod {c},\\ (a,c)=1}} \e{\frac {m\overline a}{c/p}} T_{m/ p}(a,c/p,\alpha,\chi;Y/p)&\ll p(qm/pT)^\frac12 (Yc)^\eps,
}
and so $E_{m,p}(\alpha,\chi;N)\ll (pqmT)^{1/2} N^{-1+\eps}$. Inserting such bound in~\eqref{ppf} we obtain Theorem~\ref{cuorecuf} also in the prime powers case.

\end{document}